\documentclass[10pt,a4paper,draft]{article}

\usepackage[leqno]{amsmath}
\usepackage{amssymb,amsthm,upref,amscd}
\usepackage[T1]{fontenc}
\usepackage{times}
\usepackage{cancel}
\usepackage{cite}
\usepackage{amsfonts}
\usepackage[colorinlistoftodos]{todonotes}

\newtheorem{Thm}{Theorem}[section]

\newtheorem{Lem}[Thm]{Lemma}

\newtheorem{Prop}[Thm]{Proposition}

\numberwithin{equation}{section}

\newcommand{\R}{\mathbb{R}}
\newcommand{\C}{\mathbb{C}}

\newcommand{\N}{\mathbb{N}}

\newcommand{\cC}{{\mathcal C}}

\newcommand{\cF}{{\mathcal F}}

\newcommand{\cL}{{\mathcal L}}

\newcommand{\cO}{{\mathcal O}}

\newcommand{\al}{\alpha}
\newcommand{\be}{\beta}
\newcommand{\ga}{\gamma}
\newcommand{\de}{\delta}
\newcommand{\la}{\lambda}
\newcommand{\ze}{\zeta}

\newcommand{\Ga}{\Gamma}
\newcommand{\Om}{\Omega}
\newcommand{\om}{\omega}
\newcommand{\vphi}{\varphi}
\newcommand{\eps}{\varepsilon}
\newcommand{\pa}{\partial}

\newcommand{\dist}{\text{\rm dist}}
\newcommand{\supp}{\text{\rm supp}}

\newcommand{\cat}{\text{\rm cat}}

\newcommand{\sign}{\text{\rm sign\,}}
\newcommand{\ham}{H_{KR}}
\newcommand{\conf}{\cF_N\Omega}

\newcommand{\abs}[1]{\lvert#1\rvert}

\newcommand\mytop[2]{\genfrac{}{}{0pt}{}{#1}{#2}}

\newcommand{\vlist}[2]{\genfrac{}{}{0pt}{1}{#1}{#2}}

\newcommand{\beq[1]}{\begin{equation}\label{eq:#1}}
\newcommand{\eeq}{\end{equation}}

\newenvironment{altproof}[1]
{\noindent
{\em Proof of {#1}}.}
{\nopagebreak\mbox{}\hfill $\Box$\par\addvspace{0.5cm}}

\begin{document}

\title{Critical points of the $N$-vortex Hamiltonian in bounded planar domains and steady state solutions of the incompressible Euler equations}

\author{\sc{Thomas Bartsch}\footnote{T.B. thanks the
Universit\'a di Roma ``La Sapienza'' for the invitation and the hospitality
during several visits.}\ \ \footnote{Supported by DAAD grant 50766047.}
\and
\sc{Angela Pistoia}\footnote{Supported by the M.I.U.R.
National Project ``Metodi variazionali e topologici nello studio di
fenomeni non lineari''.}
}

\date{}
\maketitle
\begin{abstract}
 We prove the existence of critical points of the $N$-vortex Hamiltonian
 \[
 \ham (x_1,\ldots, x_N)
 =\sum^N_{i=1}\Ga^2_i h(x_i)
    + \sum_{\mytop{i,j=1}{j\ne k}}^N \Ga_i\Ga_jG(x_i,x_j)+2\sum\limits_{i=1}^N\Ga_i\psi_0(x_i)
 \]
 in a bounded domain $\Om\subset\R^2$ which may be simply or multiply connected. Here $G$ denotes the Green function for the Dirichlet Laplace operator in $\Om$, more generally a hydrodynamic Green function, and $h$ the Robin function. Moreover $\psi_0\in\cC^1(\overline\Om)$ is a harmonic function on $\Om$. The domain need not be simply connected. We obtain new critical points $x=(x_1,\dots,x_N)$ for $N=3$ or $N=4$ under conditions on the vorticities $\Ga_i\in\R\setminus\{0\}$. These critical points correspond to point vortex equilibria of the Euler equation in vorticity form. The case $\Ga_i=(-1)^i$ of counter-rotating vortices with identical vortex strength is included. The point vortex equilibria can be desingularized to obtain smooth steady state solutions of the Euler equations for an ideal fluid. The velocity of these steady states will be irrotational except for $N$ vorticFity blobs near $x_1,\dots,x_N$.
\end{abstract}

{\bf Keywords}: vortex dynamics, point vortices, counter-rotating vortices, steady states of the Euler flow

{\bf  AMS subject classification}: 35J60, 35J25, 37J45, 76B47.


%
\section{Introduction}\label{sec:intro}
The dynamics of $N$ point-vortices $x_1, \ldots, x_N\in\Om$ in a bounded domain $\Om\subset\R^2$ in the plane is governed by a Hamiltonian system
\begin{equation}\label{HS}
 \left\{
  \begin{aligned}
   \Ga_i\frac{d x_{i1}}{dt}
     &=\frac{\partial \ham}{\partial x_{i,2}}(x_1,\ldots, x_N);\\
   \Ga_i\frac{dx_{i2}}{dt}
     &=-\frac{\partial \ham}{\partial x_{i1}}(x_1,\ldots, x_N);
\end{aligned}
\hspace{2cm} i=1,\ldots, N.
\right.
\end{equation}
Here $\Ga_i\in\R\setminus\{0\}$ denotes the strength of the $i$-th vortex $x_i$, the sign determining the orientation of the vortex. The Hamiltonian is given by the Kirchhoff-Routh path function
\begin{equation}\label{H_Omega}
\ham(x_1,\ldots, x_N)
 =\sum_{i=1}^N\Ga^2_i h(x_i)
    + \sum_{\mytop{i,j=1}{i\ne j}}^N \Ga_i\Ga_jG(x_i,x_j) +2\sum\limits_{i=1}^N\Ga_i\psi_0(x_i)
\end{equation}
where
\[
G(x,y)=g(x,y) - \frac{1}{2\pi}\log|x-y|
\]
is the Green function of the Dirichlet Laplacian in $\Om$. Here $g:\Om\times\Om\to\R$ is the regular part, and $h:\Om\to\R$, $h(x)=g(x,x)$, denotes the Robin function. Moreover $\psi_0\in\cC^1(\overline\Om)$ is a harmonic function on $\Om$ modeling the boundary flux. In case of a solid boundary one has $\psi_0=0$.
$\ham$ is defined on the configuration space
\[
\conf
 = \left\{(x_1,\ldots,x_N)\in\Om^N:
           x_i\neq x_j\ \text{for }i\neq j\right\}.
\]
The domain $\Om$ need neither be simply connected nor symmetric. More generally, $G$ can be a hydrodynamic Green function (see \cite{flucher-gustafsson:1997}), or even a function having certain properties of Green functions.

Based on first ideas of Helmholtz \cite{helmholtz:1858} about vortices, the system has been deduced by Kirchhoff \cite{kirchhoff:1876}, Routh \cite{routh:1881}, and Lin \cite{lin:1941a,lin:1941b} from the Euler equations
\beq[euler]
\left\{
\begin{aligned}
v_t+(v\cdot\nabla)v&=-\nabla P\\
\nabla\cdot v&=0
\end{aligned}
\right.
\eeq
for an incompressible and non-viscous fluid in $\Om$. Here $v$ denotes the velocity field and $P$ the pressure of the fluid. The scalar vorticity
$\om = \nabla\times v = \pa_1v_2-\pa_2v_1$ satisfies the equation
\beq[euler-vort]
\om_t+v\cdot\nabla\om=0.
\eeq
The point vortex ansatz $\om=\sum_{k=1}^{N}\Ga_k\de_{x_k}$, where $\de_{x_k}$ is the usual Dirac delta, leads to \eqref{HS} for the point vortices $x_k(t)$. We refer to \cite{flucher-gustafsson:1997, majda-bertozzi:2001, marchioro-pulvirenti:1994, newton:2001, saffman:1992} for modern treatments of vorticity methods.

There are many results about point vortex dynamics if $\Om=\R^2$ is the plane, or if $\Om$ is a special domain like the disc, the half-disc, an annulus, an infinite strip. In these cases the Green function, hence the Hamiltonian, is either explicitely known or one has good representations of it. There are also many results of numerical nature, due to the multiple applications of point vortex methods in science and engineering. We just refer to the surveys \cite{aref:2008,aref-etal:2002,newton:2001} and the literature cited therein.

In this paper we present new conditions on the vortex strengths $\Ga_i$ such that $\ham$ has a critical point. Our results extend considerably earlier ones from
\cite{bartolucci-pistoia:2007,bartsch-pistoia-weth:2010,esposito-etal:2006}
where only special cases have been treated, all dealing with $\Ga_i\in\{\pm 1\}$ and $\psi_0=0$. Observe that $\conf\subset\Om^N$ is an open bounded subset of $\R^{2N}$, and that $\ham$ is singular and not bounded from above nor below. Therefore the existence of critical points is highly nontrivial, in particular since we require no symmetry nor any geometrical or topological properties of the domain. Our results hold for functions $F:\conf\to\R$ which are $\cC^1$-close to $\ham$ on certain compact subsets of $\conf$. This allows to apply the methods from Cao, Liu and Wei \cite{cao-liu-wei:2013,cao-liu-wei:2014} on the desingularization of stationary point vortex solutions and to obtain stationary solutions of the Euler equations \eqref{eq:euler}, \eqref{eq:euler-vort}. This is done by constructing families $\psi_\eps$ of stream functions with vortex blobs which converge as $\eps\to0$ towards the stationary point vortices we construct. The velocity $v$ will be irrotational outside these vortex blobs. 

The paper is organized as follows. First, in Section~\ref{sec:results} we state our main results Theorems~\ref{thm:N=2} to \ref{thm:N=4} about the existence of critical points of {\it Hamiltonians of the $N$-vortex type}, and we state in Theorem~\ref{thm:euler} our results about solutions of the incompressible Euler equations. Next, in Section~\ref{sec:com} we prove a compactness result for the class of Hamiltonians we consider. This is very technical but in a sense the core of our paper.  Section~\ref{sec:proofs} contains the proofs of Theorems~\ref{thm:N=2} to \ref{thm:N=4}. Finally in Section~\ref{sec:euler} we desingularize the stationary point vortex solutions by proving Theorem~\ref{thm:euler}.

\section{Statement of results}\label{sec:results}

Let $\Om\subset\R^2$ be a bounded domain with $\cC^2$-boundary. We fix $\eps_0>0$ small so that the reflection at $\pa\Om$ is well defined in
$\Om_0:=\{x\in\Om:\dist(x,\pa\Om) < \eps_0\}$ and maps to the complement of $\Om$; we denote it by
$\Om_0 \to \R^2\setminus\overline\Om$, \ $x\mapsto\bar x$. It is of class $\cC^1$ since $\pa\Om$ is of class $\cC^2$. We write
\[
p:\Om_0\to\pa\Om,\quad p(x)=\frac12(x+\bar x),
\]
for the orthogonal projection onto the boundary, and
\[
\nu:\Om_0\to\R^2,\quad \nu(x)=\frac{1}{|x-\bar x|}(x-\bar x),
\]
for the interior normal; more precisely, $\nu(x)$ is the interior unit normal at $p(x)\in\pa\Om$ for $x\in\Om_0$. Clearly, $p(x)=x-\dist(x,\pa\Om)\nu(x)$ and
$\bar x=x-2\dist(x,\pa\Om)\nu(x)$.

Let $N\ge2$ and $\Ga_1,\dots,\Ga_N\in\R\setminus\{0\}$ be given. We consider a {\it Hamiltonian of the $N$-vortex type}, i.~e.\ a function $H:\conf\to\R$ of the form
\beq[H]
H(x)
 = \sum^N_{i=1}\Ga^2_i h(x_i) + \sum_{\mytop{i,j=1}{j\ne k}}^N\Ga_i\Ga_jG(x_i,x_j) + f(x)
\eeq
where $f\in\cC^1(\overline\Om^N)$ and
\beq[green]
G(x,y)=g(x,y) - \frac{1}{2\pi}\log|x-y|
\eeq
is a {\it generalized Green's function} by which we mean that the following properties hold.

\begin{itemize}
\item[(A1)] $G$ is bounded from below and symmetric, i.~e.\ $G(x,y)=G(y,x)$.
\item[(A2)] $g:\Om\times\Om\to\R$ is a $\cC^1$-function, bounded from above, and $h(x)=g(x,x) \to -\infty$ as $\dist(x,\pa\Om)\to0$.
\item[(A3)] For every $\eps>0$ there is a constant $C_1=C_1(\Om,\eps)>0$ such that
\[
|h(x)|+|\nabla h(x)| \le C_1 \qquad
\text{for every $x\in\Om$ with $\dist(x,\pa\Om)\ge\eps$}
\]
and
\[
|G(x,y)|+|\nabla_x G(x,y)|+|\nabla_y G(x,y)| \le C_1
\quad \text{for every $x,y \in \Om$ with $|x-y|\ge\eps$.}
\]
\item[(A4)] There exists a constant $C_2=C_2(\Om)>0$ such that
$\psi(x,y) := g(x,y)-\frac{1}{2\pi}\log|\bar x-y|$ satisfies
\[
|\psi(x,y)| + |\nabla_x \psi(x,y)| + |\nabla_y \psi(x,y)| \leq C_2
\qquad\text{for every $x,y\in\Om_0$.}
\]
\end{itemize}

It is well known that these assumptions hold for the Dirichlet Green's function, more generally for a hydrodynamic Green's function (see \cite{flucher-gustafsson:1997} for the definition); details can be found in \cite{bartsch-pistoia-weth:2010, kuhl:2014a}. Our first theorem deals with a rather simple case.

\begin{Thm}\label{thm:N=2}
  Suppose $N=2$ and $\Ga_1\Ga_2<0$. There exists a compact subset $K\subset\cF_2(\Om)$ and $\de>0$ such that the following holds:
  \begin{itemize}
  \item[a)] Any $\cC^1$-function $F:\cF_2(\Om)\to\R$ with $\|F|_K-H|_K\|_\infty<\de$ has at least $\cat(\cF_2(\Om))$ critical points $(x^i_1, x^i_2)$, $i=1,\dots,\cat(\cF_2(\Om))$ in $K$.
  \item[b)] If $\Ga_1=-\Ga_2$ and if $F$ is symmetric, i.~e.\ $F(x,y)=F(y,x)$, then $F$ has at least $k:=\cat(\cF_2(\Om)/(x_1,x_2)\sim(x_2,x_1))$ pairs $(x^i_1, x^i_2)$, $(x^i_2, x^i_1)$ of critical points in $K$, $i=1,\dots,k$.
  \item[c)] If $F_\eps:\cF_2(\Om)\to\R$ is a family of $\cC^1$-functions such that $\|F_\eps|_K-H|_K\|_{\cC^1} \to 0$ then the critical points $x_\eps$ obtained in a) or b) converge along a subsequence towards a critical point of $H$.
  \end{itemize}
\end{Thm}

Here $\cat$ denotes the Lusternik-Schnirelman category. The problem becomes considerably more difficult if $N>2$. We only deal with the cases $N=3$, $N=4$ and require the following assumption:

\begin{equation}\label{eq:hyp-1}
\begin{aligned}
&\Ga_i\Ga_{i+1}<0\ \text{for } i=1,\dots,N-1, \text{ and }\\
&\text{for every subset $I\subset\{1,\dots,N\}$ with $|I|\ge3$ there holds $\sum\limits_{i,j\in I,i\ne j} \Ga_i\Ga_j <0$.}
\end{aligned}
\end{equation}

\begin{Thm}\label{thm:N=3}
Let $N=3$ and assume \eqref{eq:hyp-1}. Then there exists a compact subset $K\subset\cF_3(\Om)$ and $\de>0$ such that the following holds:
  \begin{itemize}
  \item[a)]  Any $\cC^1$-function $F:\cF_3(\Om)\to\R$ with $\|F|_K-H|_K\|_{\cC^1}<\de$ has a critical point in $K$.
  \item[b)] If $F_\eps:\cF_2(\Om)\to\R$ is a family of $\cC^1$-functions such that $\|F_\eps|_K-H|_K\|_{\cC^1} \to 0$ then the critical points $x_\eps$ obtained in a) converge along a subsequence towards a critical point of $H$.
  \end{itemize}
\end{Thm}

In the case $N=4$ we need an additional hypothesis on the vorticities:
\begin{equation}\label{eq:hyp-2}
|\Ga_2| < |\Ga_1|+|\Ga_3|\quad\text{and}\quad |\Ga_3| < |\Ga_2|+|\Ga_4|.
\end{equation}

\begin{Thm}\label{thm:N=4}
Let $N=4$ and assume \eqref{eq:hyp-1}, \eqref{eq:hyp-2}. Then there exists a compact subset $K\subset\cF_4(\Om)$ and $\de>0$ such that the following holds:
  \begin{itemize}
  \item[a)]  Any $\cC^1$-function $F:\cF_4(\Om)\to\R$ with $\|F|_K-H|_K\|_{\cC^1}<\de$ has a critical point in $K$.
  \item[b)] If $F_\eps:\cF_2(\Om)\to\R$ is a family of $\cC^1$-functions such that $\|F_\eps|_K-H|_K\|_{\cC^1} \to 0$ then the critical points $x_\eps$ obtained in a) converge along a subsequence towards a critical point of $H$.
  \end{itemize}
\end{Thm}

Observe that \eqref{eq:hyp-1} and \eqref{eq:hyp-2} hold if $\Ga_i=(-1)^i$. This case has already been treated in \cite{bartsch-pistoia-weth:2010}. The proof of \cite[Theorem~1.2]{bartsch-pistoia-weth:2010} has a gap, however, which is being fixed in this paper using a different method though. Related results concerning point vortex equilibria on general bounded domains can also be found in \cite{kuhl:2014a} and, if the domain is symmetric, in \cite{kuhl:2014b}. These papers complement our results in that different conditions on the set of vorticities are considered. Earlier results dealing with the case of $\Om$ not being simply connected and all $\Ga_i=1$ can be found in \cite{delpino-etal:2005,esposito-etal:2005}. Periodic solutions of (HS) for any given $N$ with all $\Ga_i=1$, on bounded and unbounded domains, have been constructed in \cite{bartsch-dai:2014}.

The point vortex equilibria obtained in Theorems~\ref{thm:N=2}-\ref{thm:N=4} can be regularized as limits of vorticity distributions of smooth steady state solutions of the incompressible Euler equations in the following way. Let $G$ be the Green function of $-\Delta$ in $\Om$ with homogeneous Dirichlet boundary conditions and let $\psi_0\in\cC^2(\overline\Om)$ be harmonic in $\Om$. We consider the Kirchhoff-Routh path function $H_{KR}:\conf\to\R$ defined by
\beq[path-fct]
H_{KR}(x)
 = \sum^N_{i=1}\Ga^2_i h(x_i) + \sum_{\mytop{i,j=1}{i\ne j}}^N\Ga_i\Ga_jG(x_i,x_j) +
    2\sum^N_{i=1}\Ga_i \psi_0(x_i).
\eeq
We write $\frac{\pa\psi_0}{\pa\tau}:\pa\Om\to\R^2$ for the tangential derivative of $\psi_0$ on $\pa\Om$, and we set $(w_1,w_2)^\perp = J(w_1,w_2) := (w_2,-w_1)$.

\begin{Thm}\label{thm:euler}
Consider one of the cases
\begin{itemize}
\item[(i)] $N=2$ and $\Ga_1\Ga_2<0$;
\item[(ii)] $N=3$ and \eqref{eq:hyp-1} holds;
\item[(iii)] $N=4$ and \eqref{eq:hyp-1}, \eqref{eq:hyp-2} hold.
\end{itemize}
Then for $\eps>0$ small there exists a stationary solution $v_\eps:\Om\to\R^2$ of \eqref{eq:euler} with pressure $P_\eps$ and boundary flux $v(x)\cdot\nu(x) = \frac{\pa\psi_0(x)}{\pa\tau}$. Moreover, the scalar vorticity of $v_\eps$ is of the form $\om_\eps = \nabla\times v_\eps = \sum_{i=1}^N \om_{i,\eps}$ with $\supp(\om_{i,\eps})\to x^*_i\in\Om$ as $\eps\to0$ along a subsequence, $\int_\Om \om_{i,\eps} \to \Ga_i$, where $(x^*_1,\dots,x^*_N)\in\conf$ is a critical point of the Kirchhoff-Routh path function $H_{KR}$ from \eqref{eq:path-fct}.
\end{Thm}

Here $\supp(\om_{i,\eps})\to x_i\in\Om$ means that for $\de>0$ the support  $\supp(\om_{i,\eps})$ is contained in the $\de$-neighborhood of $x_i\in\Om$ provided $\eps$ is small. Theorem~\ref{thm:euler} will be proved by the method of stream functions. Recall that a stream function $\psi:\Om\to\R$ for $v$ satisfies
$v = J\nabla\psi = (-\pa\psi/\pa x_2,\pa\psi/\pa x_1)$, hence $\om = -\Delta\psi$ and
$v = J\nabla(-\Delta)^{-1}\om$. If $\psi:\Om\to\R$ satisfies
\beq[stream]
\left\{
\begin{aligned}
-\Delta \psi &= F'(\psi) &&\text{for $x\in\Om$,}\\
\psi &= \psi_0 &&\text{for $x\in\pa\Om$,}
\end{aligned}
\right.
\eeq
for some arbitrary function $F\in\cC^2(\R)$ then $v=J\nabla\psi$ solves \eqref{eq:euler} with pressure field $P=F(\psi)-\frac12|\nabla\psi|^2$ and vorticity $F'(\psi)$. Using the method from \cite{cao-liu-wei:2013,cao-liu-wei:2014} and our Theorems \eqref{thm:N=2}--\eqref{thm:N=4} there are appropriate functions $F_\eps$ and solutions of \eqref{eq:stream} with $F=F_\eps$ which will yield Theorem~\ref{thm:euler}. The theorems from \cite{cao-liu-wei:2013,cao-liu-wei:2014} cannot be applied directly because there it is assumed that the Kirchhoff-Routh path function $H_{KR}$ has an isolated stable critical point. This will not be the case in general, for instance, it doesn't hold for $\Om$ a disc or an annulus. The latter case is excluded in \cite{cao-liu-wei:2013} anyway because there the domain is required to be simply connected. This is needed when one wants to prescribe the boundary flux, not the function $\psi_0$.

\section{A compactness result}\label{sec:com}
We fix a function $G$ as in \eqref{eq:green} such that (A1)--(A4) hold, we fix a function $f\in\cC^1(\overline\Om)$, and we consider a Hamiltonian $H$ as in \eqref{eq:H}. Then we introduce the function $\Phi:\conf\to\R$ defined by
\[
\Phi(x) := \sum_{i=1}^N \Ga_i^2h(x_i)-\sum_{\mytop{i,j=1}{i\not=j}}^N |\Ga_i\Ga_j|G(x_i,x_j).
\]
Assumptions (A1) and (A2) imply
\[
\lim_{x\to\pa\conf}\Phi(x)=-\infty.
\]

\begin{Prop}\label{prop:com}
Assume that $N\in\{3,4\}$ and \eqref{eq:hyp-1} is satisfied. Then for any $a,b\in\R$ with $a<b$ there exists $M_0>0$ such that the following holds:
\[
\Phi(x)\le-M_0,\ a\le H(x)\le b,\ \nabla H(x) = \la\nabla\Phi(x)
\qquad\Longrightarrow\qquad \lambda>0.
\]
\end{Prop}

The rest of this section is concerned with the proof of Proposition~\ref{prop:com}. We argue by contradiction. Suppose there exist $a,b\in\R$ with $a<b$, a sequence of points $x^n=(x_1^n,\ldots, x_N^n)\in \conf$, and a sequence  $\la_n\le0$ such that
\begin{equation}\label{eq:xn}
\Phi(x^n)\to-\infty,\ a\le H(x^n)\le b, \text{ and }
 \nabla H(x^n)=\la _n\nabla \Phi(x^n).
\end{equation}
Recall from Section~\ref{sec:results} the reflection $x\mapsto\bar x$ at the boundary, the projection $x\mapsto p(x)$ onto the boundary, and the interior normal $x\mapsto\nu(x)$. These maps are defined for $x\in\Om_0$ close to the boundary. We set $d_i^n:=\dist(x_i^n,\pa\Om)$, and $\nu_i^n:=\nu(x_i^n)$, $p_i^n:=p(x_i^n)$, if $x_i^n\in\Om_0$. In the sequel $O(1)$, $o(1)$ refer to $n\to\infty$. The following lemma holds for all sequences $(x^n)_n$ in $\conf$.

\begin{Lem}\label{lem:dni}
\begin{itemize}
\item[(i)] $h(x_i^n) = \frac{1}{2\pi}\log2d_i^n + O(1)$ and $d_i^n|\nabla h(x_i^n)|=O(1)$ if $x_i^n\in\Om_0$.
\item[(ii)] $\nabla h(x_i^n) = \frac{1}{2\pi d_i^n}\nu_i^n + o(1)$ if $d_i^n\to0$
\item[(iii)] $G(x_i^n,x_j^n) = -\frac{1}{2\pi}\log|x_i^n-x_j^n| + \frac{1}{2\pi}\log|x_i^n-\bar{x}_j^n| + O(1)$ if $x_j^n\in\Om_0$.
\item[(iv)] $G(x_i^n,x_j^n) = O(1) $ if $\liminf\frac{|x_i^n-x_j^n|}{d_i^n} > 0$.
\item[(v)] $\pa_1G(x_i^n,x_j^n)
    = -\frac{1}{2\pi}\left(\frac{x_i^n-x_j^n}{|x_i^n-x_j^n|^2}
       + \frac{\bar x_i^n-x_j^n}{|\bar x_i^n-x_j^n|^2}\right) + O(1)
    = -\frac{1}{2\pi}\left(\frac{x_i^n-x_j^n}{|x_i^n-x_j^n|^2}
       + \frac{x_i^n-\bar x_j^n}{|x_i^n-\bar x_j^n|^2}\right) + O(1)$ if $x_i^n\in\Om_0$ or $x_j^n\in\Om_0$, respectively.
\item[(vi)] $d_i^n|\nabla g(x_i^n,x_j^n)|=O(1)$ if $x_i^n\in\Om_0$.
\item[(vii)] $\langle\pa_1G(x_i^n,x_j^n),\nu^n_i\rangle
      + \langle\pa_1G(x_j^n,x_i^n),\nu^n_j\rangle
    = \frac{1}{2\pi}\left(d_i^n+d_j^n\right)\left(\frac{1}{|\bar x_i^n-x_j^n|^2}
      + \frac{1}{|\bar x_j^n-x_i^n|^2}\right)+O(1)$ if $x_i^n,x_j^n\in\Om_0$.
\item[(viii)] $|\bar x_i^n-x_j^n|^2 = |x_i^n-x_j^n|^2+4d_i^nd_j^n + o(|x_i^n-x_j^n|^2)$ if $x_i^n,x_j^n\to x^*\in\pa\Om$.
\item[(ix)] $\langle p_i^n-p_j^n,\nu_i^n\rangle = O(|x_i^n-x_j^n|^2)$ if $x_i^n,x_j^n\in\Om_0$.
\end{itemize}
\end{Lem}

\begin{proof}
These statements follow in a straightforward way from assumptions (A1)--(A4).
\end{proof}

We write the proof of Proposition~\ref{prop:com} for $N=4$. The case $N=3$ is simpler and can be deduced by forgetting all arguments which involve $x_4^n$. In the sequel we drop the notation $n\to\infty$ from all kinds of limits. The first lemma does not require hypothesis \eqref{eq:hyp-1}. It is sufficient that all $\Ga_i\ne0$ for $i=1,\dots,4$.

\begin{Lem}\label{lem:i0}
There exist indices $i_0\ne j_0$ such that $\displaystyle\liminf\frac{|x_{i_0}^n-x_{j_0}^n|}{d_{i_0}^n}\to0$.
\end{Lem}

\begin{proof}
Suppose to the contrary that $|x^n_i-x^n_j|\ge cd^n_i$ for all $i\ne j$. Then \eqref{eq:xn} implies that $d^n_k\to0$ for some $k\in\{1,\dots,4\}$. Using Lemma~\ref{lem:dni} we can estimate the energy:
\[
\begin{aligned}
H(x^n)
 &= \sum\limits_{i=1}^N\Ga_i^2h(x_i^n)
     +\sum\limits_{\vlist{i,j=1}{i\neq j}}^N\Ga_i\Ga_jG(x_i^n, x_j^n)
  =\sum\limits_{i=1}^N\Ga_i^2h(x_i^i)+O(1)\\
 &\le \Ga_k^2h(x_k^n)+O(1) = \frac{1}{2\pi}\ln d_k^n+O(1) \to -\infty.
\end{aligned}
\]
This contradicts \eqref{eq:xn}.
\end{proof}

After passing to a subsequence we may assume for each $i\in\{1,\dots,4\}$:
\begin{equation}\label{eq:i0}
\text{either } |x_i^n-x_{i_0}^n|=o(d_{i_0}^n)\quad \text{ or }\quad \liminf\frac{|x_i^n-x_{i_0}^n|}{d_{i_0}^n}>0.
\end{equation}
Setting
\[
I:=\big\{i\in\{1,\dots,4\}:|x_i^n-x_{i_0}^n|=o(d_{i_0}^n)\big\}
\]
Lemma~\ref{lem:i0} implies
\begin{equation}\label{eq:I}
\left\{
\begin{aligned}
&|I|\ge2,\ \frac{d_i^n}{d_j^n}\to1\ \text{ and }\ |x_i^n-x_j^n|=o(d_i^n)\ \text{ for }\ i,j\in I,\\
&|x_i^n-x_j^n|=o(|x_i^n-x_k^n|)\ \text{ for }\ i,j\in I,\ k\notin I.
\end{aligned}
\right.
\end{equation}

\begin{Lem}\label{lem:I1}
The only possibilities for $I$ are $\{1,3\}$ or $\{2,4\}$. Moreover $\la_n\to-1$.
\end{Lem}

\begin{proof}
We set
\[
z^n:=(z^n_1,\dots,z^n_4),\quad
z^n_i:=\begin{cases}x^n_i-x^n_{i_0}&i\in I,\\ 0&i\notin I,\end{cases}
\]
and compute, using (A1)--(A4), Lemma~\ref{lem:dni}, as well as \eqref{eq:i0} and \eqref{eq:I},
\[
\begin{aligned}
\langle\nabla H(x^n),z^n\rangle
 &= \sum_{i\in I}\Ga_i^2\langle\nabla h(x_i^n),z_i^n\rangle
    +2\sum_{i\in I}\sum_{j\ne i}\Ga_i\Ga_j
      \langle\pa_1g(x_i^n,x_j^n),z_i^n\rangle\\
 &\hspace{1cm}
  -\frac{1}{\pi}\sum_{i\in I}\sum_{j\neq i}\Ga_i\Ga_j
    \frac{\langle x_i^n-x_j^n,z_i^n\rangle}{|x_i^n-x_j^n|^2}\\
 &= -\frac{1}{\pi}\sum_{i,j\in I, i<j}\Ga_i\Ga_j
      \frac{\langle x_i^n-x_j^n,x_i^n-x_1^n\rangle}{|x_i^n-x_j^n|^2}+o(1)\\
 &= -\frac{1}{\pi}\sum_{i,j\in I, i<j}\Ga_i\Ga_j+o(1)
\end{aligned}
\]

Arguing in the same way, we also obtain that
\[
\langle\nabla \Phi(x^n), z^n\rangle
 = \frac1\pi\sum\limits_{i,j\in I, i<j} |\Ga_i\Ga_j|+o(1).
\]
Now the equation $\nabla H(x^n)=\la_n\nabla\Phi(x^n)$ implies
\[
0 \ge \la_n
  \to -\frac{\sum_{i,j\in I, i<j}\Ga_i\Ga_j}{\sum_{i,j\in I, i<j} |\Ga_i\Ga_j|}.
\]
This implies $\sum\limits_{i,j\in I, i<j}\Ga_i\Ga_j\ge0$, hence $|I|\le2$ by hypothesis \eqref{eq:hyp-1}. Now \eqref{eq:I} yields $|I|=2$, and since $\Ga_i\Ga_{i+1}<0$ we must have $I=\{1,3\}$ or $I=\{2,4\}$. We also obtain immediately $\la_n\to-1$.
\end{proof}

\begin{Lem}\label{lem:I2}
At least one of the following is true:
\begin{itemize}
\item[(i)] $I=\{1,3\}$ satisfies \eqref{eq:I} and $d_1^n\to0$.
\item[(ii)] $I=\{2,4\}$ satisfies \eqref{eq:I} and $d_2^n\to0$.
\end{itemize}
\end{Lem}

\begin{proof}
Suppose $I_1=\{1,3\}$ satisfies \eqref{eq:I} but, after passing to a subsequence, $d^n_1\ge c>0$. Since $2,4\notin I_1$ there holds $|x^n_i-x^n_j|\ge c$ for $i\in\{1,3\}$, $j\in\{2,4\}$. Now $ H(x^n)=O(1)$ implies $h(x^n_2)\to-\infty$ or $h(x^n_4)\to-\infty$, hence $d^n_2\to0$ or $d^n_4\to0$. Assuming without loss of generality $d^n_2\to0$, we consider the equation
\[
\langle\pa_{x_2}\big( H(x^n)-\la_n\Phi(x^n)\big),\nu^n_2\rangle=0.
\]
Using
$\pa_1G(x^n_2,x^n_1)=O(1)=\pa_1G(x^n_2,x^n_3)$ and $\la_n\to-1$ we deduce
\[
(1-\la_n)\frac{\Ga_2^2}{2\pi d^n_2}+(1+\la_n)\Ga_2\Ga_4\pa_1G(x^n_2,x^n_4)+O(1)=0
\]
and therefore
\[
\frac{\Ga_2^2}{2\pi} + \Ga_2\Ga_4\frac{1+\la_n}{1-\la_n}
    \frac{\langle x^n_2-x^n_4,d^n_2\nu^n_2\rangle}{|x^n_2-x^n_4|^2}
 = o(1).
\]
This implies $|x^n_2-x^n_4|=o(d^n_2)$. Then \eqref{eq:I} holds for $I_2=\{2,4\}$, and $d^n_2\to0$.
\end{proof}

Without loss of generality we may now assume $I=\{1,3\}$ and $d^n_1\to0$. Thus there holds:
\begin{equation}\label{eq:dn13}
\left\{
\begin{aligned}
&\frac{d^n_1}{d^n_3}\to1;\quad |x^n_1-x^n_3|=o(d^n_1)=o(d^n_3);\\
&x^n_1,x^n_3\to p\in\pa\Om;\quad |x^n_i-x^n_1|\ge cd^n_1 \ \text{ for } i\in\{2,4\}.
\end{aligned}
\right.
\end{equation}
After passing to a subsequence we can also assume for $i\in\{1,3\}$:
\begin{equation}\label{eq:alphabeta}
\frac{d^n_1}{|x^n_i-x^n_2|}\to\al_1;\quad \frac{d^n_2}{|x^n_i-x^n_2|}\to\al_2;\quad
\frac{d^n_1}{|x^n_i-x^n_4|}\to\be_1;\quad \frac{d^n_2}{|x^n_i-x^n_4|}\to\be_2.
\end{equation}
Clearly we have $\al_1,\al_2,\be_1,\be_2\ge0$.

\begin{Lem}\label{lem:comp1}
Fix $i\in\{1,3\}$ and suppose $x^n_2\to p\in\pa\Om$. Then there holds:
\begin{itemize}
\item[(i)] $\displaystyle
 d^n_i\frac{\langle x^n_i-x^n_2,\nu^n_i\rangle}{|x^n_i-x^n_2|^2} \to \al_1(\al_1-\al_2)$
\item[(ii)] $\displaystyle
 d^n_i\frac{\langle x^n_i-\bar x^n_2,\nu^n_i\rangle}{|x^n_i-\bar x^n_2|^2}
  \to\frac{\al_1(\al_1+\al_2)}{1+4\al_1\al_2}$
\item[(iii)] $\displaystyle
 \frac{\langle\nu^n_2,x^n_2-x^n_i\rangle}{d^n_2} \to 1-\frac{\al_1}{\al_2}$
 provided $\al_2>0$.
\item[(iv)] $\displaystyle
\frac{\langle x^n_2-\bar x^n_i,x^n_2-x^n_i\rangle}{|x^n_2-\bar x^n_i|^2}
  \to \frac{1+2\al_1(\al_2-\al_1)}{1+4\al_1\al_2}$
\end{itemize}
\end{Lem}

\begin{proof}
We compute using Lemma~\ref{lem:dni}:
\[
\begin{aligned}
d^n_i\frac{\langle x^n_i-x^n_2,\nu^n_i\rangle}{|x^n_i-x^n_2|^2}
 &= d^n_i\frac{\langle d^n_i\nu^n_i-d^n_2\nu^n_2,\nu^n_i\rangle}{|x^n_i-x^n_2|^2} + o(1)\\
 &= \frac{|d^n_i|^2}{|x^n_i-x^n_2|^2} -
    \frac{d^n_id^n_2}{|x^n_i-x^n_2|^2}\left(1+\langle\nu^n_2-\nu^n_i,\nu^n_i\rangle\right)
    + o(1)\\
 &\to \al_1(\al_1-\al_2),
\end{aligned}
\]
This proves (i). Next, (ii) follows from:
\[
\begin{aligned}
d^n_i\frac{\langle x^n_i-\bar x^n_2,\nu^n_i\rangle}{|x^n_i-\bar x^n_2|^2}
 &= d^n_i
    \frac{\langle d^n_i\nu^n_i+d^n_2\nu^n_2,\nu^n_i\rangle} {|x^n_i-x^n_2|^2+4d^n_id^n_2+o(|x^n_i-x^n_2|^2)}
    + o(1)\\
 &\to\frac{\al_1(\al_1+\al_2)}{1+4\al_1\al_2}\,.
\end{aligned}
\]
In order to see (iii) we calculate:
\[
\begin{aligned}
\frac{\langle\nu^n_2,x^n_2-x^n_i\rangle}{d^n_2}
 &= \frac{\langle\nu^n_2,p^n_2-p^n_i\rangle}{|x^n_2-\bar x^n_i|}\cdot
    \frac{|x^n_2-\bar x^n_i|}{d^n_2}
   + \frac{\langle\nu^n_2,d^n_2\nu^n_2-d^n_i\nu^n_i\rangle}{d^n_2}\\
 &= o(1)\cdot\frac{1}{\al_2} + 1 - \frac{d^n_i}{d^n_2}
    + \frac{d^n_i\langle\nu^n_2,\nu^n_2-\nu^n_i\rangle}{d^n_2}\\
 &\to 1-\frac{\al_1}{\al_2}
\end{aligned}
\]
Finally we prove (iv):
\[
\begin{aligned}
\frac{\langle x^n_2-\bar x^n_i,x^n_2-x^n_i\rangle}{|x^n_2-\bar x^n_i|^2}
 &= \frac{|x^n_2-\bar x^n_i|^2+\langle 2d^n_i\nu^n_i,x^n_2-x^n_i\rangle}
     {|x^n_2-\bar x^n_i|^2+4d^n_id^n_2+o(|x^n_2-x^n_i|^2)}\\
 &= \frac{|x^n_2-\bar x^n_i|^2+2d^n_i\langle\nu^n_i,d^n_2\nu^n_2-d^n_i\nu^n_i
           +o(|x^n_2-x^n_i|^2\rangle}
     {|x^n_2-\bar x^n_i|^2+4d^n_id^n_2+o(|x^n_2-x^n_i|^2)}\\
 &\to \frac{1+2\al_1(\al_2-\al_1)}{1+4\al_1\al_2}.
\end{aligned}
\]
\end{proof}

We also need the following equality:
\begin{equation}\label{eq:pa13}
\begin{aligned}
0 &= \langle\pa_{x_1}\left( H(x^n)-\la_n\Phi(x^n)\right),\nu^n_1\rangle
     + \langle\pa_{x_3}\left( H(x^n)-\la_n\Phi(x^n)\right),\nu^n_3\rangle\\
  &= (1-\la_n)\bigg(\frac{\Ga_1^2}{2\pi d^n_1} + \frac{\Ga_3^2}{2\pi d^n_3}
      - 2|\Ga_1\Ga_2|\langle\pa_1G(x^n_1,x^n_2),\nu^n_1\rangle
      - 2|\Ga_1\Ga_4|\langle\pa_1G(x^n_1,x^n_4),\nu^n_1\rangle\\
  &\hspace{2cm}
      - 2|\Ga_3\Ga_2|\langle\pa_1G(x^n_3,x^n_2),\nu^n_3\rangle
      - 2|\Ga_3\Ga_4|\langle\pa_1G(x^n_3,x^n_4),\nu^n_3\rangle\bigg)\\
  &\hspace{2cm}
      + 2(1+\la_n)\Ga_1\Ga_3\left(\langle\pa_1G(x^n_1,x^n_3),\nu^n_1\rangle
      + \langle\pa_1G(x^n_3,x^n_1),\nu^n_3\rangle\right)
\end{aligned}
\end{equation}

\begin{Lem}\label{lem:xn24}
$x^n_2\to p$ and $x^n_4\to p$ where $p\in\pa\Om$ is from \eqref{lem:I1}.
\end{Lem}

\begin{proof}
Suppose $|x^n_2-x^n_1|\ge c>0$ and $|x^n_4-x^n_1|\ge c>0$ along a subsequence, hence $\pa_1G(x^n_i,x^n_j)=O(1)$ for $i\in\{1,3\}$, $j\in\{2,4\}$. Multiplying \eqref{eq:pa13} by $\frac{2\pi d^n_1}{1-\la_n}$, and using Lemma~\ref{lem:dni}, \eqref{eq:dn13} and $\la_n\to-1$, we obtain the contradiction:
\[
\begin{aligned}
0 &= \Ga_1^2+\Ga_3^2\frac{d^n_1}{d^n_3}
      + 2\frac{1+\la_n}{1-\la_n}\Ga_1\Ga_3d^n_1(d^n_1+d^n_3)
      \left(\frac{1}{|\bar x^n_1-x^n_3|^2}+\frac{1}{|\bar x^n_3-x^n_1|^2}\right)+o(1)\\
  &\to \Ga_1^2+\Ga_3^2.
\end{aligned}
\]
Therefore we may assume that $x^n_2\to p$. Suppose $|x^n_1-x^n_4|\ge c>0$ along a subsequence, hence $\pa_1G(x^n_i,x^n_4)=O(1)$ for $i\in\{1,3\}$. As above we multiply \eqref{eq:pa13} by $\frac{2\pi d^n_1}{1-\la_n}$ and obtain:
\[
\begin{aligned}
0 &= \Ga_1^2+\Ga_3^2\frac{d^n_1}{d^n_3}
     + 2|\Ga_1\Ga_2|d^n_1\left(\frac{\langle x^n_1-x^n_2,\nu^n_1\rangle}{|x^n_1-x^n_2|^2}
     -\frac{\langle x^n_1-\bar x^n_2,\nu^n_1\rangle}{|x^n_1-\bar x^n_2|^2}\right)\\
&\hspace{1cm}
     + 2|\Ga_3\Ga_2|d^n_1\left(\frac{\langle x^n_3-x^n_2,\nu^n_3\rangle}{|x^n_3-x^n_2|^2}
     -\frac{\langle x^n_3-\bar x^n_2,\nu^n_3\rangle}{|x^n_3-\bar x^n_2|^2}\right) + o(1)
\end{aligned}
\]
Passing to the limit now implies:
\begin{equation}\label{eq:x1}
\Ga_1^2+\Ga_3^2+2|\Ga_2|(|\Ga_1|+|\Ga_3|)\al_1
 \left(\al_1-\al_2-\frac{\al_1+\al_2}{1+4\al_1\al_2}\right) = 0.
\end{equation}
We used Lemma~\ref{lem:comp1} for this computation. Observe that \eqref{eq:x1} implies $\al_1,\al_2>0$.

We also have
\[
\begin{aligned}
0 &= \langle\pa_{x_2}\left( H(x^n)-\la_n\Phi(x^n)\right),\nu^n_2\rangle\\
  &= (1-\la_n)\left(\frac{\Ga_2^2}{2\pi d^n_2}
      - 2|\Ga_1\Ga_2|\langle\pa_1G(x^n_2,x^n_1),\nu^n_2\rangle
      - 2|\Ga_2\Ga_3|\langle\pa_1G(x^n_2,x^n_3),\nu^n_2\rangle\right)\\
  &\hspace{2cm}
      + 2(1+\la_n)\Ga_2\Ga_4\langle\pa_1G(x^n_2,x^n_4),\nu^n_2\rangle.
\end{aligned}
\]
Since we know $x^n_1,x^n_2,x^n_3\to p$ and since we are assuming $|x^n_1-x^n_4|\ge c>0$ we have $\pa_1G(x^n_2,x^n_4)=O(1)$. Therefore multiplying the above equation by
$\frac{2\pi d^n_2}{1-\la_n}$ we obtain as before
\[
\begin{aligned}
0 &= \Ga_2^2
     + 2|\Ga_1\Ga_2|d^n_2\left(\frac{\langle x^n_2-x^n_1,\nu^n_2\rangle}{|x^n_2-x^n_1|^2}
     -\frac{\langle x^n_2-\bar x^n_1,\nu^n_2\rangle}{|x^n_2-\bar x^n_1|^2}\right)\\
&\hspace{1cm}
     + 2|\Ga_2\Ga_3|d^n_2\left(\frac{\langle x^n_2-x^n_3,\nu^n_2\rangle}{|x^n_2-x^n_3|^2}
     -\frac{\langle x^n_2-\bar x^n_3,\nu^n_2\rangle}{|x^n_2-\bar x^n_3|^2}\right) + o(1).
\end{aligned}
\]
Again we pass to the limit and deduce:
\begin{equation}\label{eq:x2}
\Ga_2^2+2|\Ga_2|(|\Ga_1|+|\Ga_3|)\al_2
 \left(\al_2-\al_1-\frac{\al_1+\al_2}{1+4\al_1\al_2}\right) = 0.
\end{equation}
As before we used Lemma~\ref{lem:comp1} for this computation. We need one more equation which comes from
\[
\begin{aligned}
0 &= \langle\pa_{x_2}\left( H(x^n)-\la_n\Phi(x^n)\right),x^n_2-x^n_1\rangle\\
  &= (1-\la_n)\bigg(\frac{\Ga_2^2\langle\nu^n_2,x^n_2-x^n_1\rangle}{2\pi d^n_2}
      - 2|\Ga_1\Ga_2|\langle\pa_1G(x^n_2,x^n_1),x^n_2-x^n_1\rangle\\
  &\hspace{1cm}
      - 2|\Ga_2\Ga_3|\langle\pa_1G(x^n_2,x^n_3),x^n_2-x^n_1\rangle\bigg)
      + 2(1+\la_n)\Ga_2\Ga_4\langle\pa_1G(x^n_2,x^n_4),x^n_2-x^n_1\rangle.
\end{aligned}
\]
Since $\pa_1G(x^n_2,x^n_4)=O(1)$ we get
\[
\begin{aligned}
0 &= \Ga_2^2\frac{\langle\nu^n_2,x^n_2-x^n_1\rangle}{d^n_2}
     - 2|\Ga_1\Ga_2|\left(-\frac{\langle x^n_2-x^n_1,x^n_2-x^n_1\rangle}{|x^n_2-x^n_1|^2}
     + \frac{\langle x^n_2-\bar x^n_1,x^n_2-x^n_1\rangle}{|x^n_2-\bar x^n_1|^2}\right)\\
&\hspace{1cm}
     - 2|\Ga_2\Ga_3|\left(-\frac{\langle x^n_2-x^n_3,x^n_2-x^n_1\rangle}{|x^n_2-x^n_3|^2}
     + \frac{\langle x^n_2-\bar x^n_3,x^n_2-x^n_1\rangle}{|x^n_2-\bar x^n_3|^2}\right)
     + o(1).
\end{aligned}
\]
Passing to the limit yields
\begin{equation}\label{eq:x3}
\Ga_2^2\left(1-\frac{\al_1}{\al_2}\right)
 + 4|\Ga_2|(|\Ga_1|+|\Ga_3|)\al_1\frac{\al_1+\al_2}{1+4\al_1\al_2} = 0.
\end{equation}
The system \eqref{eq:x1}, \eqref{eq:x2}, \eqref{eq:x3} has no solutions because
$\al_2\cdot\eqref{eq:x1}+\al_1\cdot\eqref{eq:x2}+\al_2\cdot\eqref{eq:x3}$ leads to
$(\Ga_1^2+\Ga_2^2+\Ga_3^2)\al_2=0$ which contradicts $\Ga_i\ne0$, $\al_2>0$.
\end{proof}

Now we use \eqref{eq:pa13} again. The same arguments as in the derivation of \eqref{eq:x1} lead to
\begin{equation}\label{eq:y1}
\begin{aligned}
&\Ga_1^2+\Ga_3^2+2|\Ga_2|(|\Ga_1|+|\Ga_3|)\al_1
 \left(\al_1-\al_2-\frac{\al_1+\al_2}{1+4\al_1\al_2}\right)\\
&\hspace{1cm}
 +2|\Ga_4|(|\Ga_1|+|\Ga_3|)\be_1
 \left(\be_1-\be_2-\frac{\be_1+\be_2}{1+4\be_1\be_2}\right)= 0.
\end{aligned}
\end{equation}
The additional term involving $\be_1,\be_2$ comes from the fact that $x^n_4\to p$. In the derivation of \eqref{eq:x1} we assumed $|x^n_4-x^n_1|\ge c>0$. This implies $\be_1=0$, hence \eqref{eq:x1} is a special case of \eqref{eq:y1}. We need to distinguish two cases:

{\sc Case 1:} $\liminf\frac{|x^n_2-x^n_4|}{d^n_2} = 0$

{\sc Case 2:} $\liminf\frac{|x^n_2-x^n_4|}{d^n_2} > 0$

In {\sc Case 1}, after passing to a subsequence we may assume that $|x^n_2-x^n_4|=o(d^n_2)$. This implies $\frac{d^n_2}{d^n_4}\to 1$, $\be_1=\al_1$ and $\be_2=\al_2$. Therefore \eqref{eq:y1} reduces to
\begin{equation}\label{eq:z1}
\Ga_1^2+\Ga_3^2+2(|\Ga_2|+|\Ga_4|)(|\Ga_1|+|\Ga_3|)\al_1
 \left(\al_1-\al_2-\frac{\al_1+\al_2}{1+4\al_1\al_2}\right)
= 0.
\end{equation}
Moreover, in {\sc Case 1} \eqref{eq:I} is also satisfied for $I=\{2,4\}$. Thus we obtain
\begin{equation}\label{eq:z2}
\Ga_2^2+\Ga_4^2+2(|\Ga_1|+|\Ga_3|)(|\Ga_2|+|\Ga_4|)\al_2
 \left(\al_2-\al_1-\frac{\al_1+\al_2}{1+4\al_1\al_2}\right)
= 0
\end{equation}
in the same way as \eqref{eq:z1}. We need one more equation which comes from
\[
0 = \langle\pa_{x_2}\left( H(x^n)-\la_n\Phi(x^n)\right),\nu^n_2\rangle
     + \langle\pa_{x_4}\left( H(x^n)-\la_n\Phi(x^n)\right),\nu^n_4\rangle
\]
Similar computations as before lead to
\begin{equation}\label{eq:z3}
\big(\Ga_2^2+\Ga_4^2\big)\left(1-\frac{\al_1}{\al_2}\right)
 + 2\big(|\Ga_2|+|\Ga_4|\big)\big(|\Ga_1|+|\Ga_3|\big)2\al_1
 \frac{\al_1+\al_2}{1+4\al_1\al_2}
= 0
\end{equation}
Now $\al_2\cdot\eqref{eq:z1}+\al_1\cdot\eqref{eq:z2}+\al_2\cdot\eqref{eq:z3}$ leads to
$(\Ga_1^2+\Ga_2^2+\Ga_3^2+\Ga_4^2)\al_2=0$ which contradicts $\Ga_i\ne0$, $\al_2>0$.

In {\sc Case 2} we have $|x^n_2-x^n_4|\ge cd^n_2$ and $|x^n_2-x^n_4|\ge cd^n_4$. This implies
\[
\langle\pa_1G(x^n_2,x^n_4),x^n_2-x^n_1\rangle = O(1)
 = \langle\pa_1G(x^n_4,x^n_2),x^n_2-x^n_1\rangle.
\]
Then the equation
\[
\langle\pa_{x_2}\left( H(x^n)-\la_n\Phi(x^n)\right),x^n_2-x^n_1\rangle = 0
\]
leads to
\begin{equation}\label{eq:y2}
\Ga_2^2\left(1-\frac{\al_1}{\al_2}\right)
 + 2|\Ga_2|\big(|\Ga_1|+|\Ga_3|\big)2\al_1
 \frac{\al_1+\al_2}{1+4\al_1\al_2}
= 0.
\end{equation}
Analogously, the equation
\[
\langle\pa_{x_4}\left( H(x^n)-\la_n\Phi(x^n)\right),x^n_4-x^n_1\rangle = 0
\]
leads to
\begin{equation}\label{eq:y3}
\Ga_4^2\left(1-\frac{\be_1}{\be_2}\right)
 + 2|\Ga_4|\big(|\Ga_1|+|\Ga_3|\big)2\be_1
 \frac{\be_1+\be_2}{1+4\be_1\be_2}
= 0.
\end{equation}
Finally the equations \[
\langle\pa_{x_4}\left( H(x^n)-\la_n\Phi(x^n)\right),\nu^n_4\rangle = 0
\]
and
\[
\langle\pa_{x_4}\left( H(x^n)-\la_n\Phi(x^n)\right),\nu^n_4\rangle = 0
\]
lead, respectively, to
\begin{equation}\label{eq:y4}
\Ga_2^2+2|\Ga_2|(|\Ga_1|+|\Ga_3|)\al_2
 \left(\al_2-\al_1-\frac{\al_1+\al_2}{1+4\al_1\al_2}\right) = 0.
\end{equation}
and
\begin{equation}\label{eq:y5}
\Ga_4^2+2|\Ga_4|(|\Ga_1|+|\Ga_3|)\al_2
 \left(\be_2-\be_1-\frac{\be_1+\be_2}{1+4\be_1\be_2}\right) = 0.
\end{equation}
Now the sum
$\al_2\be_2\cdot\eqref{eq:y1}+\al_2\be_2\cdot\eqref{eq:y2}+\al_1\be_2\cdot\eqref{eq:y3}
+\al_2\be_1\cdot\eqref{eq:y4}+\al_2\be_2\cdot\eqref{eq:y5}$ leads to
$(\Ga_1^2+\Ga_2^2+\Ga_3^2+\Ga_4^2)\al_2\be_2=0$ which as before contradicts $\Ga_i\ne0$, $\al_2,\be_2>0$.
This concludes the proof of Proposition~\ref{prop:com}.

\section{Proof of Theorems~\ref{thm:N=2}-\ref{thm:N=4}}\label{sec:proofs}
\begin{altproof}{Theorem~\ref{thm:N=2}}
There exists a compact subset $K_0\subset\cF_2\Om$ such that $\cat(K_0)=\cat(\cF_2\Om)$. Observe that
\[
 H(x)=\Ga_1^2h(x_1)+\Ga_2^2h(x_2) + \Ga_1\Ga_2G(x_1,x_2)+f(x)\to-\infty \qquad
\text{as }x\to\pa\cF_2\Om
\]
because $\Ga_1\Ga_2<0$, $f(x)=O(1)$, and assumption (A2). Therefore
$H^{\ge a}=\{x\in\cF_2\Om: H(x)\ge a\}$ is compact for any $a\in\R$. Now we choose
$a<\min H(K_0)$, set $\de:=\frac12(\min H(K_0)-a)$, and consider $F$ on the compact manifold $K= H^{\ge a}$ with boundary $B = H^{-1}(a)$. Since $\min F(K) > \max F(B)$ standard critical point theory yields that a function $F\in\cC^1(\conf)$ with $\|F|_K-H|_K\| < \de$ has at least $\cat(\cF_2\Om)$ critical points in $K$. This proves a).

Part b) follows similarly upon passing to the quotient $\cF_2(\Om)/(x_1,x_2)\sim(x_2,x_1)$. Finally, c) is obvious.
\end{altproof}

The proof of Theorem~\ref{thm:N=3} and of Theorem \ref{thm:N=4} will be based on a linking argument. In the sequel $N$ will be either $3$ or $4$. Suppose there exists a (sequentially) compact topological space $S$, a continuous map $\ga_0:S\to\conf$, and a subset $\cL \subset \conf$ such that
\begin{equation}\label{eq:bound}
\sup_{x\in\cL} H(x) < \infty,
\end{equation}
and
\begin{equation}\label{eq:linking}
 \text{$\ga$ is homotopic to $\ga_0$}\qquad\Longrightarrow\qquad \ga(S)\cap\cL\ne\emptyset.
\end{equation}
As usual, $\ga$ being homotopic to $\ga_0$ means that there exists a continuous deformation $H:S\times[0,1]\to\conf$ with $H(\ze,0)=\ga_0(\ze)$ and $H(\ze,1)=\ga(\ze)$ for all $\ze\in S$. We shall prove that if a function $F\in\cC^1(\conf,\R)$ is close to $ H$ on compact sets then it has a critical point. In order to express the closeness we choose
$a < \min_{\ze\in S} H(\ga_0(\zeta))$ and $b > \sup_{x\in\cL} H(x)$. Let $M_0$ be as in Proposition~\ref{prop:com} for these values $a<b$. By Sard's theorem we may assume that $-M_0$ is a regular values of $\Phi$. Since $S$ is sequentially compact we may also assume that $-M_0 < \inf_{\ze\in S}\Phi(\ze)$. Setting
\[
V_\Om(x):=\nabla H(x)
 - \frac{\langle\nabla H(x),\nabla\Phi(x)\rangle}{|\nabla\Phi(x)|^2}
    \nabla\Phi(x)
\]
Proposition~\ref{prop:com} implies
\[
a\le H(x)\le b,\ \Phi(x)=-M_0,\ \langle\nabla H(x),\nabla\Phi(x)\rangle \le 0
\qquad\Longrightarrow\qquad V_\Om(x)\ne0.
\]
Observe that $D:=\{x\in\conf:\Phi(x)\ge-M_0\}$ is a compact manifold with smooth boundary $\pa D=\Phi^{-1}(-M_0)$. We also define
\[
D_a^b:=\{x\in D:a\le H(x)\le b\} = \{x\in\conf:\Phi(x)\ge-M_0,\ a\le H(x)\le b\}.
\]
Now we choose $\eps>0$ satisfying
\[
a+2\eps < \min_{\ze\in S} H(\ga_0(\zeta)) \le \sup_{x\in\cL} H(x) < b-2\eps
\]
and
\begin{equation}\label{eq:eps}
\eps < \frac12\min\{|V_\Om(x)|: a\le H(x)\le b,\ \Phi(x)=-M_0,\
         \langle\nabla H(x),\nabla\Phi(x)\rangle \le 0\}.
\end{equation}

\begin{Prop}\label{prop:crit}
Suppose $F\in\cC^1(\conf,\R)$ satisfies
\begin{equation}\label{eq:F-cond1}
|F(x)- H(x)|\le\eps\quad\text{if }x\in K:=D_a^b,
\end{equation}
and
\begin{equation}\label{eq:F-cond2}
|\nabla F(x)-\nabla H(x)|\le\eps\quad\text{if }x\in D_a^b\cap\pa D.
\end{equation}
Then $F$ has a critical point in $K=D_a^b$.
\end{Prop}

Clearly \eqref{eq:F-cond1} requires $F$ to be $\cC^0$-close to $H$ on the compact set $D_a^b$, and \eqref{eq:F-cond2} requires $F$ to be $\cC^1$-close to $H$ on the compact set
$D_a^b\cap\pa D$.

\begin{proof}
We assume that $F$ has no critical value in $D_a^b$. First we define a continuous map $V_0:D_a^b\cap\pa D\to\R^{2N}$ by setting:
\[
V_0(x):=\begin{cases}
  \nabla F(x) - \frac{\langle\nabla F(x),\nabla\Phi(x)\rangle}{|\nabla\Phi(x)|^2}
      \nabla\Phi(x)
 &\text{if }\langle\nabla F(x),\nabla\Phi(x)\rangle\le0;\\
  \nabla F(x) &\text{else.}
  \end{cases}
\]
Clearly we have
\begin{equation}\label{eq:V_0}
\langle V_0(x),\nabla\Phi(x)\rangle \ge 0\quad\text{for all } x\in D_a^b\cap\pa D,
\end{equation}
hence $V_0(x)$ is either tangent to $\pa D$ at $x$ or points inside $D$. Using \eqref{eq:eps} and \eqref{eq:F-cond2} it is easy to check that
\begin{equation}\label{eq:F-pseudo}
\langle\nabla F(x),V_0(x)\rangle > 0\quad
 \text{if }x\in D_a^b\cap\pa D.
\end{equation}
Next we extend this vector field to all of $\conf$. In order to do this we first choose a relatively open tubular neighborhood $\pa D\subset\cO\subset D$ of $\pa D$ and a diffeomorphism $\chi=(\chi_1,\chi_2):\cO\to\pa D\times [0,1)$ such that $\chi(x)=(x,0)$ for $x\in\pa D$. Then we define for $0<\de<1$ a map $V_1:D_a^b\to\R^{2N}$ by setting
\[
V_1(x):=\begin{cases}
\frac{\de-\chi_2(x)}{\de}V_0(\chi_1(x))+\frac{\chi_2(x)}{\de}\nabla F(x)\quad
&\text{if }x\in D_a^b\cap\cO,\ \chi_2(x)\le\de;\\
\nabla F(x)&\text{if }x\in D_a^b\cap\cO,\ \chi_2(x)>\de, \text{ or }x\in D_a^b\setminus\cO.
\end{cases}
\]
Observe that $V_1$ is continuous and coincides with $V_0$ on $D_a^b\cap\pa D$. Therefore, if $\de>0$ is small \eqref{eq:F-pseudo} implies that
\begin{equation}\label{eq:V}
\langle\nabla F(x),V_1(x)\rangle > 0 \quad \text{if }x\in D_a^b.
\end{equation}
Here we also used that $F$ has no critical point in $D_a^b$. We fix such a $\de>0$. Then we replace the continuous vector field $V_1:D_a^b\to\R^{2N}$ by a Lipschitz continuous vector field $V_F:D_a^b\to\R^{2N}$ such that \eqref{eq:V_0} and \eqref{eq:V} continue to hold for $V_F$ instead of $V_0$, $V_1$. Finally we extend the vector field $V_F:D_a^b\to\R^{2N}$ to a Lipschitz continuous vector field $V_F:\conf\to\R^{2N}$ such that $V_F(x)=0$ outside a neighborhood of $D_a^b$, and such that $\langle\nabla F(x),V_F(x)\rangle \ge 0$ for all $x\in\conf$. As a consequence, $V_F$ defines a global flow $\vphi:\conf\times\R\to\conf$ which satisfies:
\begin{equation}\label{eq:vphi-1}
x\in D,\ a\le H(\vphi(x,t))\le b\text{ for }0\le t\le T
\qquad\Longrightarrow\qquad
\vphi(x,T)\in D
\end{equation}
and
\begin{equation}\label{eq:vphi-2}
\left\{
\begin{aligned}
&x\in D,\ a\le H(\vphi(x,t))\le b\text{ for all } t\ge 0,\\
&\hspace{1cm}\qquad\Longrightarrow\qquad
 \vphi(x,t_n))\to \bar x\text{ for some sequence } t_n\to\infty,\ \nabla F(\bar x)=0.
\end{aligned}
\right.
\end{equation}
Now we argue as follows. By \eqref{eq:linking} for each $n\in\N$ there exists $\ze_n\in S$ such that $\vphi(\ga_0(\ze_n),n)\in\cL$, hence $a\le  H(\ga_0(\ze_n),n)\le b$. Since $S$ is sequentially compact we have $\ze_n\to\ze\in S$ along a subsequence. It follows that $x:=\ga_0(\ze)\in D_a^b$ satifies $\vphi(x,t)\in D_a^b$ for all $t\ge0$. Now the existence of a critical point of $F$ in $D_a^b$ follows from \eqref{eq:vphi-2}.
\end{proof}

In the proof of Theorems~\ref{thm:N=3} and \ref{thm:N=4} the set $\cL$ will be
\[
\cL_3(\Om):=\{x\in\cF_3\Om: x_1-x_2+r(x_3-x_2)=0 \text{ for some }r>0\},
\]
in case $N=3$, and the set
\[
\cL_4(\Om):=\{x\in\cF_4\Om: x_1-x_2+r(x_3-x_2)=0,\ x_2-x_3+s(x_4-x_3)=0 \text{ for some }r,s>0\}.
\]
in case $N=4$, as in \cite{bartsch-pistoia-weth:2010}. So we need to bound $H$ on these sets.

\begin{Lem}\label{lem:3d}
a) Suppose $N=3$ and \eqref{eq:hyp-1} holds. Then $\sup_{\cL_3\Om} H<\infty$.

b) Suppose $N=4$ and \eqref{eq:hyp-1} and \eqref{eq:hyp-2} hold. Then $\sup_{\cL_4\Om} H<\infty$.
\end{Lem}

\begin{proof}
We shall prove that if $x^n\in\cL_N\Om$ is such that $x^n\to\pa\cL_N\Om$, then $H(x^n)\to-\infty$. Set $d^n_i:=\dist(x^n_i,\pa\Om)$. As in Section~\ref{sec:com} we drop the notation $n\to\infty$ from all limits, in particular for the terms $O(1)$ and $o(1)$. Without loss of generality we may assume that $\Ga_i=(-1)^ik_i$ with $k_i>0$.

a) The Hamiltonian has the form
\[
H(x)=\sum_{i=1}^3k_i^2h(x_i)-2k_1k_2G(x_1,x_2)+2k_1k_3G(x_1,x_3)-2k_2k_3G(x_2,x_3)+f(x),
\]
and assumption \eqref{eq:hyp-1} reads as
\beq[hyp-1a]
k_1k_2+k_2k_3-k_1k_3>0.
\eeq
Observe that if $x\in \cL_3\Om$ then
\begin{equation}\label{l31}
|x_1-x_3|>\max\{|x_1-x_2|,|x_2-x_3|\}.
\end{equation}

If $d^n_i\ge c>0$ for every $i$ then $h(x^n_i),g(x^n_i,x^n_j)=O(1)$, and $|x^n_i-x^n_j|\to0$ for at least one $i\ne j$. Then \eqref{eq:hyp-1a} and \eqref{l31} imply
\[
\begin{aligned}
H(x^n)
&= \sum_{i=1}^3k_i^2h(x^n_i)-2k_1k_2g(x^n_1,x^n_2)+2k_1k_3g(x^n_1,x^n_3)-2k_2k_3g(x^n_2,x^n_3)\\ &\hspace{1cm}
   +\frac1\pi\log\frac{|x^n_1-x^n_2|^{k_1k_2}|x^n_2-x^n_3|^{k_2k_3}}{|x^n_1-x^n_3|^{k_1k_3}} + f(x^n)\\
& = \frac1\pi\log\frac{|x^n_1-x^n_2|^{k_1k_2}|x^n_2-x^n_3|^{k_2k_3}}{|x^n_1-x^n_3|^{k_1k_3}} + O(1)
 \to -\infty
\end{aligned}
\]

Thus we may assume from now on that $d^n_i\to0$ for some $i$. If in addition
$|x^n_j-x^n_\ell|\ge c>0$ for every $j\ne\ell$ then $H(x^n)\to-\infty$ by (A2) and because $f(x^n)=O(1)$. It follows that we only need to consider the case where $|x^n_j-x^n_\ell|\to0$ for some $j$ and $\ell$. Observe that if only one of $|x^n_1-x^n_2|\to0 $ or $|x^n_2-x^n_3|\to0$ hold then  (A1) and (A2) immediately imply $H(x^n)\to-\infty$. Therefore we may assume that $|x^n_1-x^n_3| \to 0$ and $d^n_i \to 0$ for some $i$, hence $d^n_i \to 0$ for every $i$ because of \eqref{l31}. Thus we are left with the following case:
\[
|x^n_1-x^n_3| \to 0,\ d^n_i \to 0 \text{ for all }i.
\]

If  $\frac{|x^n_1-x^n_3|}{d^n_1}\ge c>0$ Lemma~\ref{lem:dni} (iv) implies $G(x^n_1,x^n_3)=O(1)$, and the claim follows. Therefore it remains to consider the case $|x^n_1-x^n_3|=o(d^n_1)$, hence also $|x^n_1-x^n_3|=o(d^n_3)$. By \eqref{l31} we also have $|x^n_1-x^n_2|=o(d^n_1)$, and $d^n_i/d^n_j\to1$ for all $i,j$. Furthermore we can deduce that $|x^n_1-\bar x^n_2|,|x^n_3-\bar x^n_2|\ge d^n_2=d^n_1(1+o(1))$, and $|x^n_1-\bar x^n_3|\le c d^n_1$. Now (i) and (iii) of Lemma~\ref{lem:dni} yield
\[
\begin{aligned}
2\pi H(x^n)
&= \log\left((d^n_1)^{k_1^2}(d^n_2)^{k_2^2}(d^n_3)^{k_3^2}
     \frac{|x^n_1-x^n_2|^{2k_1k_2}|x^n_1-\bar x^n_3|^{2k_1k_3}|x^n_2-x^n_3|^{2k_2k_3}}
       {|x^n_1-\bar x^n_2|^{2k_1k_2}|x^n_1-x^n_3|^{2k_1k_3}|\bar x^n_2- x^n_3|^{2k_2k_3}}
        \right)
    + O(1)\\
 &\le \log\left((d^n_1)^{k_1^2}(d^n_2)^{k_2^2}(d^n_3)^{k_3^2}
       \frac{|x^n_1-x^n_3|^{2k_1k_2+2k_2k_3-2k_1k_3}|x^n_1-\bar x^n_3|^{2k_1k_3}}
            {(d^n_2)^{2k_1k_2+2k_2k_3}}\right) + O(1)\\
 &\le \log\left(c (d^n_1)^{k_1^2}(d^n_2)^{k_2^2}(d^n_3)^{k_3^2}
                \left(\frac{|x^n_1-x^n_3|}{d^n_1}\right)^{2k_1k_2+2k_2k_3-2k_1k_3}\right) + O(1)
 \to -\infty, \end{aligned}
\]
for some constant $c>0$. For the convergence we used assumption \eqref{eq:hyp-1a} and $|x^n_1-x^n_3|=o(d^n_1)$.

b) Here the Hamiltonian has the form
\[
\begin{aligned}
H(x)
 &= \sum_{i=1}^4 k_i^2h(x_i) - 2k_1k_2G(x_1,x_2) + 2k_1k_3G(x_1,x_3) - 2k_1k_4G(x_1,x_4)\\
 &\hspace{1cm} - 2k_2k_3G(x_2,x_3) + 2k_2k_4G(x_2,x_4) - 2k_3k_4G(x_3,x_4) + f(x).
\end{aligned}
\]
Assumption \eqref{eq:hyp-1} implies
\beq[hyp-1b]
k_1k_2+k_2k_3-k_1k_3>0,\quad k_2k_3+k_3k_4-k_2k_4>0,
\eeq
and assumption \eqref{eq:hyp-2} implies
\beq[hyp-2b]
k_1(k_2 + k_4 - k_3)>0,\quad k_4(k_1 + k_3 - k_2) > 0.
\eeq
For $x\in \cL_4\Om$ there holds
\begin{equation}\label{l41}
\begin{aligned}
&|x_1-x_3| > \max\{|x_1-x_2|,|x_2-x_3|\},\quad |x_2-x_4| > \max\{|x_2-x_3|,|x_3-x_4|\}\\
&|x_1-x_4| > \max\{|x_1-x_3|,|x_2-x_4|\}.
\end{aligned}
\end{equation}

If $|x^n_j-x^n_\ell|\ge c>0$ for every $j\ne\ell$ then $d^n_i\to0$ for some $i$ and $H(x^n)\to-\infty$ as a consequence of (A1) and (A2). If $|x^n_j-x^n_\ell|\to0$ for some $j\ne\ell$ then the only case we have to check is when $|x^n_1-x^n_4|\to0 $ because all the other cases can be treated as in the proof of a). In this case, if $d^n_i\ge c>0$ for every $i$ we have
\[
\begin{aligned}
H(x^n)
&= \frac1\pi\log
 \frac{|x^n_1-x^n_2|^{k_1k_2}|x^n_1-x^n_4|^{k_1k_4}|x^n_2-x^n_3|^{k_2k_3}|x^n_3-x^n_4|^{k_3k_4}}
   {|x^n_1-x^n_3|^{k_1k_3}|x^n_2-x^n_4|^{k_2k_4}} + O(1) \to -\infty,
\end{aligned}
\]
because for some $c>0$
\[
\begin{aligned}
&\frac{|x^n_1-x^n_2|^{k_1k_2}|x^n_1-x^n_4|^{k_1k_4}|x^n_2-x^n_3|^{k_2k_3}|x^n_3-x^n_4|^{k_3k_4}}
      {|x^n_1-x^n_3|^{k_1k_3}|x^n_2-x^n_4|^{k_2k_4}} \\
&\hspace{1cm}
 \le c\left(|x^n_1-x^n_3|^{k_1k_4} + |x^n_3-x^n_4|^{k_1k_4}\right)
        \frac{|x^n_1-x^n_2|^{k_1k_2}|x^n_2-x^n_3|^{k_2k_3}|x^n_3-x^n_4|^{k_3k_4}}
             {|x^n_1-x^n_3|^{k_1k_3}|x^n_2-x^n_4|^{k_2k_4}} \\
&\hspace{1cm}
 \le c{|x^n_1-x^n_3|^{k_1k_2+k_1k_4-k_1k_3}}{|x^n_2-x^n_4|^{k_2k_3+k_3k_4-k_2k_4}}\\
&\hspace{2cm}
      + c{|x^n_1-x^n_3|^{k_1k_2+k_2k_3-k_1k_3}}{|x^n_2-x^n_4|^{k_1k_4+k_3k_4-k_2k_4}}\\
&\hspace{1cm}
 \to 0\,.
\end{aligned}
\]
Here we used \eqref{eq:hyp-1b}, \eqref{eq:hyp-2b}, and \eqref{l41}.

It remains to consider the case when  $|x^n_1-x^n_4|\to0$ and $d^n_i\to0$ for some $i$ which implies $d^n_i\to0$ for every $i$ by \eqref{l41}. If
\beq[13d1]
\frac{|x^n_1-x^n_3|}{d^n_1} \ge c>0
\eeq
and
\beq[24d2]
\frac{|x^n_2-x^n_4|}{d^n_2} \ge c>0
\eeq
Lemma~\ref{lem:dni} (iv) implies $G(x^n_1,x^n_3)=O(1)$ and $G(x^n_2,x^n_4)=O(1)$ and the claim follows. If only one of \eqref{eq:13d1}, \eqref{eq:24d2} is true we argue as in a).

Finally, we are left with the case $|x^n_1-x^n_3|=o(d^n_1)$ and $|x^n_2-x^n_4|=o(d^n_2)$. In this case, it is easy to check that
\begin{equation}\label{ext}
|x^n_1-x^n_4|=o(d^n_1)\quad \hbox{and}\quad\frac{d^n_i}{d^n_j} \to 1.
\end{equation}
Setting $q^n_{i,j} := \frac{|x^n_i-x^n_j|}{|x^n_i-\bar x^n_j|}$, Lemma \ref{lem:dni} (i), (ii) yields
\[
2\pi H(x^n)
 = \log(d^n_1)^{k_1^2}(d^n_2)^{k_2^2}(d^n_3)^{k_3^2}(d^n_4)^{k_4^2}
    + \log\frac{q^n_{1,2}q^n_{1,4}q^n_{2,3}q^n_{3,4}}{q^n_{1,3}q^n_{2,4}} + O(1).
\]
From \eqref{eq:hyp-2b}, \eqref{l41}, \eqref{ext}, we deduce
$d^n_j \le |x^n_i - \bar x^n_j| \le 3d^n_j$, hence $q^n_{1,2} \le \frac{c|x^n_1-x^n_3|}{d^n_1}$ for some $c>0$, and similarly for the other $q^n_{i,j}$. Using this and
$|x^n_1-x^n_4| \le |x^n_1-x^n_3| + |x^n_2-x^n_4|$ we obtain
\[
\begin{aligned}
\frac{q^n_{1,2}q^n_{1,4}q^n_{2,3}q^n_{3,4}}{q^n_{1,3}q^n_{2,4}}
&\le c\left(\frac{|x^n_1-x^n_3|}{d^n_1}\right)^{2k_1k_2-2k_1k_3+2k_1k_4}
       \left(\frac{|x^n_2-x^n_4|}{d^n_2}\right)^{2k_2k_3+2k_3k_4-2k_2k_4} \\
&\hspace{1cm}
 + c\left(\frac{|x^n_1-x^n_3|}{d^n_1}\right)^{2k_1k_2+2k_2k_3-2k_1k_2}
     \left(\frac{|x^n_2-x^n_4|}{d^n_2}\right)^{2k_1k_4+2k_3k_4-2k_2k_4} \\
&\to 0
\end{aligned}
\]
Thus also in this case $H(x^n)\to-\infty$.
\end{proof}

\begin{altproof}{Theorem~\ref{thm:N=3}}
We recall the linking from \cite{bartsch-pistoia-weth:2010}. We assume without loss of generality that $0\in\Om$ and fix $\rho>0$ such that the closed ball $B(0,2\rho)\subset\Om$. Using complex notation for the elements of $\Om\subset\R^2=\C$, we set
\begin{equation}\label{eq:ga0_N=3}
\gamma_0:S^1=\{\zeta\in\C:|\zeta|=1\}\to\cF_3\Om,\quad
\gamma_0(\zeta):=(\rho\zeta,0,2\rho).
\end{equation}
Then \eqref{eq:linking} holds for $S=S^1$, $\ga_0$ from \eqref{eq:ga0_N=3}, and $\cL=\cL_3(\Om)$. This has been proved in \cite[Lemma~6.2]{bartsch-pistoia-weth:2010}. It follows that a $\cC^1$-function $F:\cF_3\Om\to\R$ which is $\cC^1$-close to $H$ in the sense of Proposition~\ref{prop:crit} has a critical point. This proves part a) of  Theorem~\ref{thm:N=3}, part b) is proved easily.
\end{altproof}

\begin{altproof}{Theorem~\ref{thm:N=4}}
For $N=4$ vortices we set
\begin{equation}\label{eq:ga0_N=4}
\gamma_0:S^1\times S^1\to\cF_4\Om, \quad
\gamma_0(\zeta_1,\zeta_2):=(\rho\zeta_1,0,3\rho,3\rho+\rho\zeta_2).
\end{equation}
It has been proved in \cite[Lemma~7.2]{bartsch-pistoia-weth:2010} that \eqref{eq:linking} holds for $S=S^1\times S^1$, $\ga_0$ from \eqref{eq:ga0_N=4}, and $\cL=\cL_4(\Om)$. As above it follows that a $\cC^1$-function $F:\cF_4\Om\to\R$ which is $\cC^1$-close to $H$ in the sense of Proposition~\ref{prop:crit} has a critical point.
\end{altproof}

\section{Proof of Theorem~\ref{thm:euler}}\label{sec:euler}
Following \cite{cao-liu-wei:2013} we prove Theorem~\ref{thm:euler} by constructing streamfunctions $\psi_\eps$ as solutions of the ellipic problem
\beq[pde]
\left\{
\begin{aligned}
-\eps^2\Delta\psi &= \sum_{i=1}^Nf_i\left(\psi+\frac{\Ga_i}{2\pi}\ln\eps\right)
 &&\text{in $\Om$;}\\
\psi &= \psi_0 &&\text{on $\pa\Om$.}
\end{aligned}
\right.
\eeq
with $f_i(t)=t_+^p$ if $\Ga_i>0$, and $f_i(t)=-t_-^p$ if $\Ga_i<0$; here
$t_\pm=\max\{\pm t,0\}$ and $1<p<\frac{N+2}{N-2}$. Setting $u=\frac{2\pi}{|\ln\eps|}(\psi-\psi_0)$ and
$\de=\eps\left(\frac{2\pi}{|\ln\eps|}\right)^{(p-1)/2}$ these are obtained as critical points of the functional $I:H^1_0(\Om)\to\R$ defined by
\beq[def-I]
I(u)
 = \frac{\de^2}{2}\int_\Om |\nabla u|^2
     -\sum_{i=1}^N\int_\Om F_i\left(u-\Ga_i-\frac{2\pi\psi_0(x)}{|\ln\eps|}\right)
\eeq
with $F_i(t)=\int_0^t f_i(s)ds$. Choose $R>0$ such that $\Om\subset\subset B_R(0)$. For $a>0$ let $W_{\de,a}$ be the unique positive solution of
\[
\left\{
\begin{aligned}
-\de^2\Delta w &= (w-a)_+^p &&\text{in $B_R(0)$;}\\
w &= 0 &&\text{on $\pa B_R(0)$,}
\end{aligned}
\right.
\]
and define $W_{\de,x,a}(y):=W_{\de,a}(y-x)$ for $x,y\in\Om$. Finally, let $P:H^1_0(B_R(0))\to H^1_0(\Om)$ be the orthogonal projection, hence $w=PW_{\de,x,a}$ solves
\[
\left\{
\begin{aligned}
-\de^2\Delta w &= (W_{\de,x,a}-a)_+^p &&\text{in $\Om$;}\\
w &= 0 &&\text{on $\pa\Om$.}
\end{aligned}
\right.
\]
Now in order to obtain a solution of \eqref{eq:pde}, for $x\in\conf$ and $a_i>0$ one makes the ansatz
\[
u=\sum_{i=1}^N(\sign\Ga_i) P W_{\de,x_i,a_i} + w_\de
\]
with $w_\de$ a small perturbation. Then a Lyapunov-Schmidt procedure yields
$w_{\de,x}\in H^1_0(\Om)$ with $\|w_{\de,x}\|_\infty=O(\de|\ln\de|^{(p-1)/2})$ and $a_{i,\de}(x)>0$ such that the following holds: If $x\in\conf$ is a critical point of
\[
F_\de(x) := I\left(\sum_{i=1}^N(\sign\Ga_i) PW_{\de,x_i,a_{i,\de}(x)} + w_{\de,x}\right)
\]
then
\beq[spikes]
u_\de = \sum_{i=1}^N(\sign\Ga_i) P W_{\de,x_i,a_{i,\de}(x)} + w_{\de,x}
\eeq
is a critical point of $I$; see \cite[Section~3]{cao-liu-wei:2013}.

By \cite[(4.2), (4.3)]{cao-liu-wei:2013} there holds
\[
F_\de(x) = \al(\de) + \be(\de)H_{KR}(x) + \chi_\de(x)
\]
where $\al(\de)$ and $\be(\de)$ are independent of $x$, and $\chi_\de$ converges to $0$ as $\de\to0$ uniformly in the $\cC^1$-norm on compact sets of $\conf$.

Now Theorems~\ref{thm:N=2}--\ref{thm:N=4} yield for $\de>0$ small critical points $x_\de\in\conf$ of $F_\de$ such that $x_\de\to x^*$ along a subsequence, where $x^*\in\conf$ is a critical point of $H_{KR}$. As a consequence we obtain corresponding critical points
$u_\de$ of $I$ as in \eqref{eq:spikes}, hence solutions $v_\de$ of the Euler equation \eqref{eq:euler}. That the scalar vorticity $\om_\de=\nabla\times v_\de=-\Delta u_\de$ concentrates near $x^*$ follows as in \cite{cao-liu-wei:2013} from the fact that $\Delta PW_{\de,x_i,a_{i,\de}(x)}=0$ if $W_{\de,x_i,a_{i,\de}(x)}<a_i$.

{\sc Address of the authors:}\\[1em]
\begin{tabular}{ll}
 Thomas Bartsch & Angela Pistoia\\
 Mathematisches Institut & Dipartimento di Metodi e Modelli Matematici\\
 University of Giessen & Universit\`{a} di Roma "La Sapienza"\\
 Arndtstr.\ 2 & via A. Scarpa 16\\
 35392 Giessen & 00161 Roma\\
 Germany & Italy\\
 Thomas.Bartsch@math.uni-giessen.de & pistoia@dmmm.uniroma1.it
\end{tabular}

\end{document}